%% file: Comparison.tex
\documentclass[12pt]{article}
\usepackage[top=3cm, bottom=3cm, left=2cm, right=2cm]{geometry}
\input{intro.tex}

\begin{document}

\title{\textbf{Cohomological Obstructions for Varieties over $p$-adic Function Fields}}
\author{Yisheng TIAN}
\date{}
\maketitle

\begin{abstract}
We introduce various cohomological obstructions for smooth integral varieties over $p$-adic function fields.
We show that the unramified obstruction is the finest one among obstructions arising from arithmetic dualities.
We also construct an explicit example whose unramified obstruction gives a proper obstruction
while Manin obstruction does not.
Finally, we compare the unramified obstruction with the descent obstruction.
\end{abstract}


\section{Introduction}
In the past decade,
there has been increasing interest in the arithmetic of varieties
defined over higher dimensional semi-global function fields.
For instance, the following works are particularly relevant to this article.
Harari, Scheiderer, Szamuely, Tian, Linh and Zhang \cites{HSS15, HS16, Tia21WA, Lin24, Zha24} studied cohomological obstructions
to the local-global principle, weak approximation and strong approximation for reductive groups or their homogenous spaces over $p$-adic function fields,
and Izquierdo \cites{Izq15, Izq17} investigated the local-global principle for function fields over higher local fields.
Also, \cites{HSS15, Tia24} showed that the degree three unramified cohomology group gives
the only obstruction to both the local-global principle and weak approximation.

Another origin of the present article is the comparison of various obstructions over number fields.
The works \cites{Har02, Sto07, Dem09, Sko09, Cao18AF, CDX19} give
explicit comparisons of obstructions arising from Brauer groups and various descents.
These developments suggest a systematic understanding of the obstruction theory over $p$-adic function fields.
Besides,
Colliot-Th\'el\`ene \cite{CT03}*{p.~174} conjectured that weak approximation with Brauer--Manin obstruction is the only one
for rationally connected smooth varieties over number fields.
This conjecture also motivates us to explore the correct analogue of the usual Brauer--Manin obstruction over $p$-adic function fields.

Now we review the main results of the article.
Let $K$ be the function field of a smooth projective geometrically integral curve $C$ over a $p$-adic field.
For a smooth integral $K$-variety $X$, let $X(\BA)$ be the set of adelic points on it.
Consider the \'etale sheaf $\QZ(2)\ce \drl\mu_n^{\ots2}$ on $X$,
see \bref{para: motivic complexes} below.
There is a well-defined Manin pairing (see (\ref{para: obstructions from ADT}\ref{para: obstructions from ADT 1}))
\[
X(\BA)\times \het^3(X,\QZ(2))\to \QZ,
\]
which defines an obstruction $X(\BA)^{\het^3}$ as its left kernel.
We have used $\het^3(X,\QZ(2))$ instead of $\Br(X)$
for cohomological dimension reasons.
In the mean time, 
the unramified cohomology group
\[
\hnr^3(X,\QZ(2))\ce H^0_{\Zar}(X,\CH^3_{\et}(\QZ(2)))
\]
also defines an obstruction $X(\BA)^{\hnr^3}$ as the left kernel of the pairing 
(see (\ref{para: unramified cohomology}\ref{para: unramified cohomology 3}))
\[
X(\BA)\times \hnr^3(X,\QZ(2))\to \QZ.
\]
On the other hand,
let $\BF:\Sch_K\to \Set$ be any contravariant functor from the category of $K$-schemes to that of sets.
For any $\al\in \BF(X)$, we obtain an associated subset $X(\BA)^{\al}\subset X(\BA)$
containing $X(K)$ (see \bref{para: obstructions from functors} below).
By taking $\BF=\het^i(-,M)$ for some commutative linear algebraic $K$-group $M$,
we arrive at the main result of \Cref{section: BM and nr} and \Cref{section: BM and desc}.

\begin{thm}\label{thm: introduction 1}
Let $X$ be a smooth integral $K$-variety and
let $M$ be a commutative linear algebraic $K$-group.
\benumr
\itm\label{thm: introduction 11}
The unramified obstruction is finer than the Manin obstruction,
i.e., $X(\BA)^{\hnr^3}\subset X(\BA)^{\het^3}$.

\itm\label{thm: introduction 12}
For any $\al\in \het^i(X,M)$ and any $i\ge 1$,
we have $X(\BA)^{\hnr^3}\subset X(\BA)^{\al}$.

\itm\label{thm: introduction 13}
Suppose that $M$ is finite.
For any $\al\in \hnr^i(X,M)\ce H^0_{\Zar}(X,\CH^i_{\et}(M))$ and any $i\ge 0$,
we have $X(\BA)^{\hnr^3}\subset X(\BA)^{\al}$.
\eenum
\end{thm}

The obstructions $X(\BA)^{\hnr^3}$ and $X(\BA)^{\het^3}$ are related by motivic complexes.
More precisely,
the short exact sequence of complexes
\[
0\to \Z(2)\to \Q(2)\to \QZ(2)\to 0
\]
yields a map $\het^3(X,\QZ(2))\to \het^4(X,\Z(2))$.
Thus the functoriality of the Manin pairing induces
a map $X(\BA)^{\het^4}\to X(\BA)^{\het^3}$ of sets.
In applying truncations as \cite{Kah12}*{\S2},
we get a map $\het^4(X,\Z(2))\to \hnr^3(X,\QZ(2))$
which enables us to define a map $X(\BA)^{\hnr^3}\to X(\BA)^{\het^4}$.
Now we may use either cohomological tools or arithmetic dualities
to conclude \ref{thm: introduction 11}.

For \ref{thm: introduction 12} and \ref{thm: introduction 13},
we can pass freely from commutative linear groups
to groups of multiplicative type because of $\tz(K)=0$.
Then a piece of a Poitou--Tate sequence for groups of multiplicative type yields the desired comparisons.
As a particular case of \ref{thm: introduction 12}, we see that the unramified obstruction is finer
than the commutative descent obstruction.

There are two natural questions after \cref{thm: introduction 1}.
Namely,
we ask for a $K$-variety $X$ such that $X(\BA)^{\hnr^3}\ne X(\BA)^{\het^3}$ and
the comparison of $X(\BA)^{\hnr^3}$ with the descent obstruction.
\Cref{section: BM and nr for versal torsors} aims to answer the first question via the following explicit example.

\begin{thm}[see \Cref{section: BM and nr for versal torsors}]
Let $G$ be an absolutely almost-simple simply connected $K$-group and
let $G\to \SL_n$ be a $K$-embedding for some $n\ge 1$.
Let $Y=\BBA^{n}_K\setminus Z$ where $Z\subset \BBA^n_K$ is a closed subvariety with $\codim(Z,\BBA^{n}_K)\ge 3$
such that $G$ acts on $Y$ freely.
Let $X\ce Y/G$.
Suppose $H^1(K_v,G)\ne 1$ and $\ker\s{R}_v=1$ for infinitely many $v\in C\uun$,
where $\s{R}$ is the Rost invariant of $G$ and $\s{R}_v:H^1(K_v,G)\to H^3(K_v,\QZ(2))$ is the map for each place $v\in C\uun$
$($see \textup{\bref{para: Rost invariant and an example from Hu}}$)$.
Then we have
\[
X(\BA)^{\hnr^3}\subsetneq X(\BA)^{\het^3}=X(\BA).
\]
\end{thm}

The first step is proving $\het^3(X,\QZ(2))\simeq H^3(K,\QZ(2))$
by the \v{C}ech spectral sequence and the Hochschild--Serre spectral sequence.
We immediately deduce $X(\BA)^{\het^3}=X(\BA)$.
The next step is based on a series of works \cites{Mer16red, Mer16ss, Mer17} of Merkurjev
which enables us to identify $\hnr^3(X,\QZ(2))$ with the group of cohomological invariants $\Inv^3(G,\QZ(2))$.
Subsequently, we conclude $X(\BA)^{\hnr^3}\subsetneq X(\BA)$ by modifying a local evaluation.

\vem
The descent obstruction is central in obstruction theory (see for instance \cites{Sko99, Sko01}).
In the context of number fields,
the descent obstruction is unconditionally finer than the Brauer--Manin obstruction,
while the latter is finer than the connected descent obstruction (see \cite{Har02}*{Th\'eor\`eme 2}).
Over $p$-adic function fields,
we shall compare the unramified obstruction with the descent obstruction.
As of today, the lack of a suitable generalization of the Kottwitz exact sequence (\cite{Bor98}*{Theorem 5.16})
limits further comparisons of these obstructions.

\begin{thm}\label{thm: intro descent vs hnr}
Let $G$ be a linear $K$-group and fix a $K$-embedding $G\to \SL_n$ for some $n\ge 2$. 
The following are equivalent.
\benumr
\itm\label{thm: intro descent vs hnr 1}
Let $X=\SL_n/G$. 
There is an identification of sets $X(\BA)^{\hnr^3}={\SL_n(\BA).X(K)}$.

\itm\label{thm: intro descent vs hnr 2}
Let $Z$ be any smooth integral variety over $K$.
For any $G$-torsor $f:Y\to Z$,
there is an inclusion $Z(\BA)^{\hnr^3}\subset Z(\BA)^f$.
\eenum
\end{thm}

The proof is somewhat standard.
Consider the contracted product ${_n}Y\ce \SL_n\times^GY$ and the induced $\SL_n$-torsor ${_n}Y\to Z$.
Subsequently, as soon as we proved $\hnr^3(Z,\QZ(2))\simeq \hnr^3({_n}Y,\QZ(2))$,
any element $z\in Z(\BA)^{\hnr^3}$ can be written as a local family in some twist of $Y$, i.e., $Z(\BA)^{\hnr^3}\subset Z(\BA)^f$.
Conversely, taking $Z=X$ and $Y=\SL_n$ implies \ref{thm: intro descent vs hnr 1}.

\begin{para}[Organization of the present article]
\Cref{section: notation and conventions} is devoted to the constructions of the Manin obstruction, the unramified obstruction and the descent obstruction
via functors and arithmetic dualities.
In Sections \ref{section: BM and nr} and \ref{section: BM and desc}, we give several comparisons of these obstructions.
Subsequently in \Cref{section: BM and nr for versal torsors}, we provide an explicit example whose Manin obstruction coincides with the whole set of adelic points
while the unramified obstruction does give a proper obstruction.
Finally in \Cref{section: unramified obs with general descent}, we compare the unramified obstruction with
the descent obstruction.
\end{para}

\begin{para}
[Acknowledgement]
The author would like to thank Yang CAO and Yong HU for helpful discussions and valuable comments.
The author is partially supported by the grant of National Natural Science Foundation of China (no. 12401014).
\end{para}

\section{Notation and conventions}\label{section: notation and conventions}
\begin{para}
[Function fields]
Let $k$ be a $p$-adic field, i.e., a finite extension of $\Qp$.
Let $C$ be a smooth projective geometrically integral curve over $k$
with function field $K$.
We write $C\uun$ for the set of closed points on $C$.
For any $v\in C\uun$, the local ring $\CO_{C,v}$ is a discrete valuation ring
(because it is a $1$-dimensional regular local ring).
Thus we write $K_v$ for the completion of $K$ \wrt any $v\in C\uun$.
\end{para}

\begin{para}
[Adelic points]
For a $K$-variety $X$ (i.e., a separated $K$-scheme of finite type),
we denote by $X(\BA)$ the set of adelic points on $X$ arising from $C$.
More precisely, we put
\[
X(\BA)\ce \drl_{U\subset C_0} \Big(\tstpl_{v\notin U}X(K_v)\times \tstpl_{v\in U}\CX(\CO_v)\Big),
\]
where $C_0\subset C$ is a sufficiently small non-empty open subset such that there exists
a smooth separated integral $C_0$-scheme $\CX$ satisfying $X\simeq \CX\times_{C_0}K$.
The set $X(\BA)$ is independent of the choice of $C_0$ and $\CX$ by a limit argument
\cite{EGAIV3}*{\S8}.
Finally, the diagonal map $X(K)\to \prod_{v\in C\uun}X(K_v)$ factors through $X(\BA)$.
\end{para}

\begin{para}
[Motivic complexes]\label{para: motivic complexes}
Let $X$ be a $K$-variety.
For each $i\ge 0$, let $z^i(X_{\tau},\bullet)$ be Bloch's cycle complex defined in \cite{Blo86}*{pp.~267-268}.
Let $\Z_X(i)_{\tau}\ce z^i(X_{\tau},\bullet)[-2i]$ be the motivic complex of $\tau$-sheaves
for $\tau\in \{\Zar, \et\}$ obtained by shifting $z^i(X_{\tau},\bullet)$.
For any abelian group $A$,
put $A(i)_{\tau}\ce A\ots\Z(i)_{\tau}$.
According to \cite{GL01}*{Theorem 1.5}, there is a quasi-isomorphism $\Z/n\Z(i)_{\et}\simeq \mu_n^{\ots i}$ of complexes.
Thus we may identify $\QZ(i)_{\et}$ with $\drl_n \mu_n^{\ots i}$.
In the sequel, we write
\[
\LB\ce \QZ(2)_{\et}
\qand
\LB(-1)\ce \QZ(1)_{\et}.
\]
If there is no risk of confusion, we write $\Z(i)$ for $\Z(i)_{\et}$.
Finally, for any field $F$ of $\tz(F)=0$, various Kummer sequences $0\to \mu_n\to \gm\to \gm\to 0$ 
induce an isomorphism
\beq\label{eq: Kummer sequence H2 and Br}
H^2(F,\LB(-1))\simeq \Br(F).
\eeq
\end{para}

\begin{para}
[Obstructions from functors]\label{para: obstructions from functors}
Let $X$ be any $K$-variety.
Let $\BF:\Sch_K\to \Set$ be a contravariant functor from the category of $K$-schemes to that of sets.
For any commutative $K$-algebra $R$ and any $\al\in \BF(X)$,
define
\[
\ev_{\al}:X(R)\to \BF(R),\ x\mt x^*(\al),
\]
where $x^*:\BF(X)\to \BF(R)$ is the map induced by $x\in X(R)$.
Consider the \cmdm of sets
\[
\xm{
X(K)\ar[r]\ar[d]_-{\ev_{\al}} & X(\BA)\ar[d]^-{\ev_{\al}} \\
\BF(K)\ar[r]^-{\Delta} & \prod\limits_{v\in C\uun}\BF(K_v).
}
\]
We put
\[
X(\BA)^{\al}\ce \{x\in X(\BA)\zc \ev_{\al}(x)\in \Im\Delta\}
\qand
X(\BA)^{\BF}=X(\BA)^{\BF(X)}\ce \tst\bigcap\limits_{\al\in \BF(X)}X(\BA)^{\al}.
\]
The commutativity of the diagram yields $X(K)\subset X(\BA)^{\al}$ and $X(K)\subset X(\BA)^{\BF}$.
Due to \cite{Poo17}*{Corollary 8.1.9},
the assignment $X\mt X(\BA)^{\BF}$ is functorial in $X$.
\end{para}

\begin{para}
[Descent obstructions]
Let $G$ be a linear $K$-group.
Let $[Y]\in H^1(X,G)$ be the class of a $G$-torsor $f:Y\to X$.
We define a descent obstruction to the local-global principle
\[
X(\BA)^f\ce \tst\bigcup\limits_{[\sg]\in H^1(K,G)}{_{\sg}}f({_{\sg}}Y(\BA)).
\]
By a twisting argument, we see that
\[
(x_v)\in X(\BA)^f
\iff
(x_v^*([Y]))\in \Im(H^1(K,G)\to \tstpl_{v\in C\uun}H^1(K_v,G)), 
\]
i.e., $X(\BA)^f$ coincides with the set $X(\BA)^{[Y]}$ defined in \bref{para: obstructions from functors}.
\end{para}

Another typical obstruction is constructed via arithmetic dualities and local invariants.

\begin{para}
[Local invariants]
For any $v\in C\uun$,
the exact sequence $0\to \Z(2)\to \Q(2)\to \LB\to 0$ induces
isomorphisms of abelian groups
\[
\xm{
H^4(K_v,\Z(2)) & H^3(K_v,\LB)\ar[l]_-{\dif_v}\ar[r]^-{j_v} & \QZ
}
\]
by \cite{HS16}*{Lemma 2.1} and
\cite{Mil06}*{Theorem 2.17} or \cite{Kat80}*{Theorem III}, respectively.
By abuse of language, both the isomorphisms
\[
j_v:H^3(K_v,\LB)\simeq \QZ
\qand
j'_v\ce j_v\circ\dif_v\ui:H^4(K_v,\Z(2))\simeq \QZ
\]
are called \emph{local invariants}.
\end{para}

\begin{para}
[Unramified cohomology groups]\label{para: unramified cohomology}
Let $X$ be a smooth integral $K$-variety with function field $K(X)$.
Let $j\ge 1$ be any integer.
\benumr
\itm
Let $F$ be a finite commutative linear group over $K$.
We put
\[
\hnr^i(X,F)\ce H^0_{\Zar}(X,\CH^i_{\et}(F)),
\]
where $\CH^i_{\et}(F)$ is the Zariski sheaf associated to the presheaf
$U\mt \het^i(U,F)$ on $X$.
So we obtain a contravariant functor $\hnr^i(-,F)$.

\itm
In practice, one frequently uses the twisted sheaf $\mu_n^{\ots j}$ instead of $\mu_n$ itself.
To this end, we also introduce
\[
\hnr^i(X,\mu_n^{\ots j})\ce H^0_{\Zar}(X,\CH^i_{\et}(\mu_n^{\ots j})),
\]
where $\CH^i_{\et}(\mu_n^{\ots j})$ is the Zariski sheaf associated to the presheaf
$U\mt \het^i(U,\mu_n^{\ots j})$.
The group $\hnr^i(X,\mu_n^{\ots j})$ has an alternative description by the following injective map
(by \cite{CT95}*{Theorem 3.8.1})
\[
i_P:\het^i(\CO_{X,P},\mu_n^{\ots j})\to H^i(K(X),\mu_n^{\ots j})
\]
for any $P\in X$.
Thus the following intersections are well-defined
\[
\tst\bigcap\limits_{P\in X\uun}\het^i(\CO_{X,P},\mu_n^{\ots j}),\
\tst\bigcap\limits_{P\in X}\het^i(\CO_{X,P},\mu_n^{\ots j})\subset H^i(K(X),\mu_n^{\ots j}).
\]
Due to \cite{CT95}*{Theorem 4.1.1~a)-c)}, we get 
\beq\label{eq: nr cohomology is E0n term}
\tst\bigcap\limits_{P\in X\uun}\het^i(\CO_{X,P},\mu_n^{\ots j})
\simeq
\hnr^i(X,\mu_n^{\ots j}) 
\simeq
\tst\bigcap\limits_{P\in X}\het^i(\CO_{X,P},\mu_n^{\ots j})
\eeq
The group $\hnr^3(X,\LB)$ is generally strictly larger than
the \emph{usual unramified cohomology group} defined as $\hnr^3(K(X)/K,\LB)\ce \bigcap \ker\dif_v$
with 
$\dif_v$ the residue map running through \emph{all} discrete valuations $v$ of $K(X)$ over $K$.
For instance,
we have
\[
\hnr^3(K(\gm)/K,\LB)/H^3(K,\LB)=0
\qand
\hnr^3(\gm,\LB)/H^3(K,\LB)\simeq \Br(K),
\]
where the former holds since $\gm$ is rational and the latter is a consequence of \cite{CT95}*{Th\'eor\`eme 3.4.1}.
If $X=\e L$ for some field $L/K$,
then $\hnr^i(X,\mu_n^{\ots j})=H^i(L,\mu_n^{\ots j})$ is the usual Galois cohomology group. 

\itm\label{para: unramified cohomology 3}
Now take $j=2$ and consider cohomologies with coefficients $\LB$.
Take any $(x_v)\in X(\BA)$ and any $\beta\in \hnr^3(X,\LB)$.
For each $v\in C\uun$, let $\iota_v:\hnr^3(X,\LB)\to \het^3(\CO_{X,P_v},\LB)$
be the inclusion map, where $P_v$ is the image of $x_v$.
Note that $x_v\in X(K_v)$ factorizes as
\[
\xm{
\e K_v\ar[r]^-{\tilde{x}_v} & \e \CO_{X,P_v}\ar[r]^-{\tau_v} & X.
}
\]
Similarly to \cite{HSS15}*{p.~2777 (16)}, we obtain the following pairing
\beq\label{eq: BM pairing for Hnr}
\lip-,-\rip_{\nr}:X(\BA)\times \hnr^3(X,\LB)\to \QZ,\
((x_v),\beta)\mt \tstsl_{v\in C\uun}j_v\circ \tilde{x}_v^*(\iota_v(\beta)),
\eeq
where $\tilde{x}_v^*:\het^3(\CO_{X,P_v},\LB)\to H^3(K_v,\LB)$ is induced by $\tilde{x}_v$.
We show that \bref{eq: BM pairing for Hnr} is a finite sum.
Choose a smooth integral model $\CX\to C_0$ over a sufficiently small non-empty open subset $C_0\subset C$.
In viewing $\beta\in \hnr^3(X,\LB)$ as an element of $H^3(k(\CX),\LB)$,
we conclude $\beta\in \het^3(\CO_{\CX,P},\LB)$ for all $P\in \CX^{(1)}$
except finitely many $\seq{P}{n}$
(because $\beta$ is unramified at all but finitely many codimension $1$ points).
Since $\beta$ is unramified over the generic fibre $X$,
by shrinking $C_0$ we see that $\beta$ is also unramified over $\CX$, i.e.,
$\seq{P}{n}$ lie in finitely many closed fibres
$\CX_{v_1},\dots,\CX_{v_r}$ of $\CX\to C_0$.
So for $v\ne v_i$ we have $\beta\in \het^3(\CO_{\CX,P},\LB)$ for all $P\in \CX_v$ and
hence $\tilde{x}_v^*(\iota_v(\beta))\in H^3(K_v,\LB)$ 
actually lies in 
$\het^3(\CO_v,\LB)\simeq H^3(\kappa(v),\LB)=0$,
which completes the proof.

\itm
The pairing \bref{eq: BM pairing for Hnr} admits a simpler description.
Each $K_v$-point $x_v\in X(K_v)$ induces a map
$\tilde{x}_v^{\#}:\hnr^3(X,\LB)\to H^3(K_v,\LB)$
by \bref{eq: nr cohomology is E0n term}. 
By construction, $\tilde{x}_v^{\#}$ factorizes as
\[
\xm{
\hnr^3(X,\LB)\ar[r]^-{\iota_v} & \het^3(\CO_{X,P_v},\LB)\ar[r]^-{\tilde{x}_v^*} & H^3(K_v,\LB),
}
\]
i.e., $\tilde{x}_v^*\circ\iota_v=\tilde{x}_v^{\#}$,
so \bref{eq: BM pairing for Hnr} can be rewritten as
\beq\label{eq: BM pairing for Hnr simpler}
\lip-,-\rip_{\nr}:X(\BA)\times \hnr^3(X,\LB)\to \QZ,\
((x_v),\beta)\mt \tstsl_{v\in C\uun}j_v\circ \tilde{x}_v^{\#}(\beta).
\eeq
\eenum
\end{para}


\begin{para}
[Obstructions from arithmetic dualities]\label{para: obstructions from ADT}
Let $X$ be a $K$-variety.
\benumr
\itm\label{para: obstructions from ADT 1}
Similarly to the usual Brauer--Manin pairing,
there is a well-defined pairing (see \cite{HS16}*{pp.~589-590}):
\beq\label{eq: BM pairing}
X(\BA)\times \het^3(X,\LB)\to \QZ,\ ((x_v),{\al})\mt \tstsl_{v\in C\uun}j_v\circ x_v^*(\al),
\eeq
where $x_v^*:\het^3(X,\LB)\to H^3(K_v,\LB)$ is the map induced by $x_v\in X(K_v)$.
The \emph{Manin obstruction} $X(\BA)^{\het^3(X,\LB)}$ is defined to be the left kernel of the pairing \bref{eq: BM pairing}.
If there is no dangerous of confusion, we write $X(\BA)^{\het^3(X)}$ or even $X(\BA)^{\het^3}$ for simplicity.
We keep the same convention in \ref{para: obstructions from ADT 2} and \ref{para: obstructions from ADT 3}.
Roughly speaking, if $X$ is a $K$-torsor under certain linear group,
\cite{HS16}*{Theorems 5.1 and 6.1} tells us that the Manin obstruction
$X(\BA)^{\het^3}$ is the only one to the local-global principle.

\itm\label{para: obstructions from ADT 2}
Since the motivic complex $\Z(2)$ arises naturally when we deal with arithmetic dualities,
it is reasonable to consider the pairing
\beq\label{eq: BM pairing H4}
X(\BA)\times \het^4(X,\Z(2))\to \QZ,\ ((x_v),\gamma)\mt \tstsl_{v\in C\uun}j_v'\circ x_v^{\#}(\gamma),
\eeq
where $x_v^{\#}:\het^4(X,\Z(2))\to H^4(K_v,\Z(2))$ is the map induced by $x_v\in X(K_v)$,
and $j_v':\het^4(K_v,\Z(2))\to \QZ$ is the local invariant.
Let $X(\BA)^{\het^4}$ be the left kernel of the pairing \bref{eq: BM pairing H4}
which is also called the \emph{Manin obstruction}.

\itm\label{para: obstructions from ADT 3}
If $X$ is smooth and integral,
the left kernel $X(\BA)^{\hnr^3}$ of \bref{eq: BM pairing for Hnr simpler} is called the \emph{unramified obstruction}.
When $X$ is a $K$-torus or a connected reductive $K$-group,
certain subgroup of $\hnr^3(X,\LB)$ can be used to measure the defect to weak approximation
(see \cite{HSS15}*{Theorem 5.2} and \cite{Tia21WA}*{Theorem 4.1}).
\eenum
\end{para}

\section{Manin obstructions and unramified obstructions}\label{section: BM and nr}
In this section, let $X$ be any smooth integral $K$-variety.
Applying \bref{para: obstructions from functors} to the functors
$\het^3(-,\LB)$, $\het^4(-,\Z(2))$ and $\hnr^3(-,\LB)$
yields respective obstructions
$X(\BA)^{\BF_{\et}^3}$, $X(\BA)^{\BF_{\et}^4}$ and $X(\BA)^{\BF_{\nr}^3}$.
Here we write $\BF$ to emphasize that
the obstructions are obtained via \emph{functors} instead of arithmetic dualities.
In the mean time, there are obstructions $X(\BA)^{\het^3}$, $X(\BA)^{\het^4}$ and $X(\BA)^{\hnr^3}$ arising from arithmetic dualities
\bref{para: obstructions from ADT}.
This section is devoted to the proof of
\[
\xm{
X(\BA)^{\BF_{\nr}^3}\ar@{=}[r]\ar@{=}[d] & X(\BA)^{\BF_{\et}^4}\ar@{^(->}[r]\ar@{=}[d] & X(\BA)^{\BF_{\et}^3}\ar@{=}[d] \\
X(\BA)^{\hnr^3}\ar@{=}[r] & X(\BA)^{\het^4}\ar@{^(->}[r] & X(\BA)^{\het^3}
}
\]
which implies \cref{thm: introduction 1}\ref{thm: introduction 11}.
Hence these are all the only obstructions to the local-global principle
if $X$ is a $K$-torsor under certain linear group
(see \cite{HS16}*{Theorems 5.1 and 6.1}).

\begin{para}
We relate the groups
$\het^3(X,\LB)$, $\het^4(X,\Z(2))$ and $\hnr^3(X,\LB)$
for a smooth integral $K$-variety $X$.
Let $q:X_{\et}\to X_{\Zar}$ be the projection from the small \'etale site to the small Zariski site.
Consider the distinguished triangles of Zariski sheaves
$\Z(2)_{\Zar}\to \BR q_*\Z(2)_{\et}\to \tau_{\ge 4}\BR q_*\Z(2)_{\et}\to \Z(2)_{\Zar}[1]$
and
$\LB_{\Zar}\to \BR q_*\LB_{\et}\to \tau_{\ge 3}\BR q_*\LB_{\et}\to \LB_{\Zar}[1]$
(see \cite{Kah12}*{\S2D-2F}).
We obtain a morphism of distinguished triangles
\[
\xm{
\LB_{\Zar}\ar[r]\ar[d] & \BR q_*\LB_{\et}\ar[r]\ar[d] & \tau_{\ge 3}\BR q_*\LB_{\et}\ar[r]\ar@{-->}[d] & \LB_{\Zar}[1]\ar[d] \\
\Z(2)_{\Zar}[1]\ar[r] & \BR q_*\Z(2)_{\et}[1]\ar[r] & \tau_{\ge 4}\BR q_*\Z(2)_{\et}[1]\ar[r] & \Z(2)_{\Zar}[2],
}
\]
where the vertical arrows are induced by the distinguished triangle
$\Z(2)_{\tau}\to \Q(2)_{\tau}\to \LB_{\tau}\to \Z(2)_{\tau}[1]$
with $\tau\in \{\Zar,\et\}$.
Taking Zariski cohomology yields a \cmdm of abelian groups with exact rows
\beac\label{diag: construct phi and psi}
\xm{
H^3_{\Zar}(X,\LB)\ar[r]\ar[d] & \het^3(X,\LB)\ar[r]\ar[d]_-{\dif_X} & H^0_{\Zar}(X,\CH^3_{\et}(\LB))\ar@{=}[d]\ar[r] & \rmCH^2(X)\ots_{\Z}\QZ\ar[d] \\
\rmCH^2(X)\ar[r] & \het^4(X,\Z(2))\ar[r] & H^0_{\Zar}(X,\CH^3_{\et}(\LB))\ar[r] & 0
}
\eeac
by \cite{Kah12}*{Propositions 2.9 and 2.11}.
Thanks to \bref{eq: nr cohomology is E0n term}, we obtain group homomorphisms
\[
\phi_X:\het^3(X,\LB)\to \hnr^3(X,\LB)
\qand
\psi_X:\het^4(X,\Z(2))\to \hnr^3(X,\LB)
\]
with $\psi_X\circ\dif_X=\phi_X$.
Note that $\psi_X$ is surjective by the exactness of the lower row in \bref{diag: construct phi and psi}.
\end{para}

\begin{prop}\label{prop: comparison of functor H3 and Hnr3}
Let $X$ be a smooth integral $K$-variety.
The unramified obstruction is finer than the Manin obstruction, i.e.,
\[
X(\BA)^{\BF_{\nr}^3}=X(\BA)^{\BF_{\et}^4}\subset X(\BA)^{\BF_{\et}^3}.
\]
\end{prop}
\begin{proof}
Consider the following \cmdm of abelian groups
\beac\label{diag: Het3 Het4 and Hnr3}
\xm{
\het^3(X,\LB)\ar[r]^-{(x_v^*)}\ar[d]_-{\dif_X} & \prod H^3(K_v,\LB)\ar[d]_-{\simeq}^-{(\dif_v)} & H^3(K,\LB)\ar[l]_-{\Delta}\ar[d]_-{\simeq}^-{\dif} \\
\het^4(X,\Z(2))\ar[r]^-{(x_v^{\#})}\ar@{->>}[d]_-{\psi_X} & \prod H^4(K_v,\Z(2))\ar[d]_-{\simeq}^-{(\psi_v)} & H^4(K,\Z(2))\ar[l]_-{\Delta}\ar[d]_-{\simeq}^-{\psi} \\
\hnr^3(X,\LB)\ar[r]^-{(\tilde{x}_v^{\#})} & \prod \hnr^3(K_v,\LB) & \hnr^3(K,\LB).\ar[l]_-{\Delta_{\nr}}
}
\eeac
Take any $\al\in \het^4(X,\Z(2))$.
Thus $(x_v^{\#}\al)\in \Im\Delta$ \ttiff $(\tilde{x}_v^{\#}(\psi_X\al))\in \Im\Delta_{\nr}$ by a diagram chase,
i.e., $X(\BA)^{\al}=X(\BA)^{\psi_X(\al)}$.
So it follows that
\[
X(\BA)^{\BF_{\nr}^3}
=\tst\bigcap\limits_{\beta\in \hnr^3(X)}X(\BA)^{\beta}
=\tst\bigcap\limits_{\al\in \het^4(X)}X(\BA)^{\psi_X(\al)}
=\tst\bigcap\limits_{\al\in \het^4(X)}X(\BA)^{\al}
=X(\BA)^{\BF^4_{\et}},
\]
where the second equality holds by the surjectivity of $\psi_X$.
Similarly, a direct computation yields $X(\BA)^{\BF_{\et}^4}\subset X(\BA)^{\BF_{\et}^3}$
via the upper squares in \bref{diag: Het3 Het4 and Hnr3}.
\end{proof}

\begin{cor}\label{cor: H3nr is H3et for curves}
Let $X$ be a smooth integral variety over $K$.
If $\rmCH^2(X)=0$, 
then we have $X(\BA)^{\BF_{\nr}^3}=X(\BA)^{\BF_{\et}^3}$.
\end{cor}
\begin{proof}
By the exactness of the rows in \bref{diag: construct phi and psi},
$\phi_X:\het^3(X,\LB)\to \hnr^3(X,\LB)$ is surjective
and $\psi_X:\het^4(X,\Z(2))\to \hnr^3(X,\LB)$ is bijective.
Hence $\dif_X$ is also surjective by $\psi_X\circ\dif_X=\phi_X$.
\end{proof}

We now compare the obstructions defined by functors with those by arithmetic dualities.
The generalized Weil reciprocity law \cite{Ser94}*{Chapitre II, Annexe, \S3}
gives us a complex
\[
\xm{
H^3(K,\LB)\ar[r] & \tst\bigoplus\limits_{v\in C\uun}H^3(K_v,\LB)\ar[r]^-{\sum j_v} & \QZ
}
\]
which is actually exact by a piece of Poitou--Tate sequence with finite coefficients.

\begin{lem}\label{lem: exactess of the generalized Weil reciprocity law}
There is an exact sequence of abelian groups
\[
H^3(K,\LB)\to \tst\bigoplus\limits_{v\in C\uun}H^3(K_v,\LB)\to \QZ\to 0.
\]
\end{lem}
\begin{proof}
Let $F=\mu_n^{\ots2}$ and let
$F'\ce \hom(\mu_n^{\ots2},\QZ(2))$.
Since $\mu_n^{\ots2}$ is $n$-torsion,
we conclude $F'=\hom(\mu_n^{\ots2},\mu_n^{\ots2})\simeq \Z/n\Z$.
Subsequently, we obtain
\[
H^0(K,F')^D\simeq \hom(\Z/n\Z,\QZ)\simeq \Z/n\Z.
\]
But the Poitou--Tate duality for $F$ gives us an exact sequence
\[
H^3(K,F)\to \tst\bigoplus\limits_{v\in C\uun}H^3(K_v,F)\to H^0(K,F')^D\to 0
\]
by \cite{HSS15}*{Theorem 3.3},
so taking direct limit yields the desired exact sequence.
\end{proof}


Now, we show that arithmetic dualities define the same obstructions as functors.

\begin{prop}\label{prop: comparison of funct obs with ADT obs}
Let $X$ be a smooth integral $K$-variety.
We have
\[
X(\BA)^{\BF_{\et}^3}= X(\BA)^{\het^3}, \quad\
X(\BA)^{\BF_{\et}^4}= X(\BA)^{\het^4}
\qand
X(\BA)^{\BF_{\nr}^3}= X(\BA)^{\hnr^3}.
\]
\end{prop}

\begin{proof}
\hfill
\bitem
\itm
The lower row of the commutative diagram below is exact by \cref{lem: exactess of the generalized Weil reciprocity law}
\beac\label{diag: Brauer--Manin with generalized reciprocity}
\xm{
X(K)\ar[r]\ar[d]_-{\ev} & X(\BA)\ar[d]^-{\ev} \\
H^3(K,\LB)\ar[r]^-{\Delta} & \bigoplus H^3(K_v,\LB)\ar[r]^-{\Sigma} & \QZ,
}
\eeac
where the diagonal map $X(\BA)\to \prod H^3(K_v,\LB)$ factors through the direct sum by
\cite{HS16}*{p.~589, (16)}.
Take any $(x_v)\in X(\BA)$ and any $\al\in \het^3(X,\LB)$.
We deduce
\[
(x_v^*\al)\in \Im\Delta=\Ker\Sigma
\iff
\lip (x_v),\al\rip=\tst\sum j_v\circ x_v^*(\al)=0
\]
by the exactness of the lower row in \bref{diag: Brauer--Manin with generalized reciprocity},
which means $X(\BA)^{\BF_{\et}^3}= X(\BA)^{\het^3}$.

\itm
Since $H^3(-,\LB)\simeq H^4(-,\Z(2))$ for $K$ and any $K_v$,
we conclude $X(\BA)^{\BF_{\et}^4}= X(\BA)^{\het^4}$.

\itm
Take any $(x_v)\in X(\BA)$ and any $\beta\in \hnr^3(X,\LB)$.
By assumption, we have
\[
(\tilde{x}_v^{\#}\beta)\in \Im\Delta=\ker\Sigma
\iff
\lip (x_v),\beta\rip_{\nr}
=\tst\sum j_v\circ x_v^{\#}(\beta)
=0
\]
by \bref{eq: BM pairing for Hnr simpler} and \bref{diag: Brauer--Manin with generalized reciprocity},
i.e., $X(\BA)^{\BF_{\nr}^3}= X(\BA)^{\hnr^3}$.
\qedhere
\eitem
\end{proof}

Although the next result is an immediate consequence of \cref{prop: comparison of functor H3 and Hnr3} and
\ref{prop: comparison of funct obs with ADT obs},
it is worth mentioning that these comparisons can be deduced from arithmetic dualities.

\begin{prop}\label{prop: Het3 Het4 and Hnr3}
Let $X$ be a smooth integral $K$-variety.
We have
\[
X(\BA)^{\hnr^3}=X(\BA)^{\het^4}\subset X(\BA)^{\het^3}.
\]
\end{prop}
\begin{proof}
Let $\dif_v, \phi_v, \psi_v$ be as in \bref{diag: Het3 Het4 and Hnr3}.
Since $\phi_v=\psi_v\circ\dif_v:H^3(K_v,\LB)\to H^3(K_v,\LB)$ is the identity,
we conclude $\psi_v=\dif_v\ui$.
\bitem
\itm
We prove $X(\BA)^{\hnr^3}=X(\BA)^{\het^4}$.
Take any $(x_v)\in X(\BA)$ and any $\al\in \het^4(X,\Z(2))$.
By the commutativity of the lower squares of \bref{diag: Het3 Het4 and Hnr3},
we obtain
\[
\tstsl_{v\in C\uun} j_v'\circ x_v^{\#}(\al)
=\tstsl_{v\in C\uun} j_v\circ \psi_v\circ x_v^{\#}(\al)
=\tstsl_{v\in C\uun} j_v\circ \tilde{x}_v^{\#}(\psi_X(\al)),
\]
i.e., $((x_v),\al)=\lip (x_v),\psi_X(\al)\rip_{\nr}$.
Since $\psi_X:\het^4(X,\Z(2))\to \hnr^3(X,\LB)$ is surjective,
$(x_v)\perp \het^4(X,\Z(2))$ amounts to $(x_v)\perp \hnr^3(X,\LB)$,
i.e., $X(\BA)^{\hnr^3}=X(\BA)^{\het^4}$.

\itm
We prove $X(\BA)^{\het^4}\subset X(\BA)^{\het^3}$.
Take any $(x_v)\in X(\BA)^{\het^4}$ and any $\al\in \het^3(X,\LB)$.
We have
$\tstsl j_v\circ x_v^*(\al)
=\tstsl j_v'\circ x_v^{\#}(\dif_X(\al))
=0$, i.e., $(x_v)\in X(\BA)^{\het^3}$.
\qedhere
\eitem
\end{proof}

Analogously to \cref{cor: H3nr is H3et for curves},
\cref{prop: Het3 Het4 and Hnr3} gives equivalent obstructions when $X$ is a curve.

\begin{cor}\label{lem: comparing H3 and H4}
If $X$ is a smooth integral curve over $K$, then $X(\BA)^{\het^4}=X(\BA)^{\het^3}$.
\end{cor}
\begin{proof}
Take any $(x_v)\in X(\BA)^{\het^3}$ and any $\gamma\in \het^4(X,\Z(2))$.
The coboundary map $\dif_X:\het^3(X,\LB)\to H^4_{\et}(X,\Z(2))$ is surjective
by the vanishing of $\het^4(X,\Q(2))=0$
(see the proof of \cite{HS16}*{Lemma 1.1}).
Thus $\gamma=\dif_X(\al)$ for some $\al\in \het^3(X,\LB)$.
\inpart we obtain
$\sum j_v'\circ x_v^{\#}(\gamma)
=\sum j_v'\circ x_v^{\#}(\dif_X(\al))
=\sum j_v'\circ \dif_v\circ\, x_v^*(\al)
=\sum j_v\circ x_v^*(\al)=0$.
\end{proof}

\section{Manin obstructions and commutative descents}\label{section: BM and desc}
In this section, we show that the commutative descent obstruction
is not fine enough.

\begin{para}
Let $M$ be a $K$-group of multiplicative type.
We fix a non-empty open subset $C_0\subset C$ such that $X$ admits a smooth integral $C_0$-model $\CX$.
We may shrink $C_0$ such that $M$ admits a finite \'etale $C_0$-model $\CM$.
For each $i\ge 0$, we put
\[
\BBP^i(K,M)\ce \drl_{U\subset C_0} \Big(\tstpl_{v\notin U}H^i(K_v,M)\times \tstpl_{v\in U}\het^i(\CO_v,\CM)\Big)
\subset \tstpl_{v\in C\uun}H^i(K_v,M).
\]
A limit argument (see \cite{EGAIV3}*{\S8}) shows that
$\BBP^i(K,M)$ is independent of the choice of $C_0$ and $\CM$.
By construction,
the image of the evaluation map $X(\BA)\to \prod H^i(K_v,M)$ lies in $\BBP^i(K,M)$.
The following theorem completes the proof of \cref{thm: introduction 1}.
\end{para}

\begin{thm}\label{thm: comparing Het and various descents}
Let $X$ be a smooth integral $K$-variety.
\benumr
\itm\label{thm: comparing Het and various descents 1}
Let $M$ be a $K$-group of multiplicative type.
For any $i\ge 1$ and $\al\in \het^i(X,M)$, we have
\[
X(\BA)^{\hnr^3}\subset X(\BA)^{\al}. 
\]

\itm\label{thm: comparing Het and various descents 2}
Let $F$ be a finite commutative $K$-group.
For any $i\ge 0$ and $\al\in \het^i(X,F)$, we have
\[
X(\BA)^{\het^3}\subset X(\BA)^{\al}. 
\]
\inpart the unramified obstruction $X(\BA)^{\hnr^3}$ is finer than $X(\BA)^{\al}$.
\eenum
\end{thm}
\begin{proof}
Take any $(x_v)\in X(\BA)$.
\benumr
\itm
Embed $M$ into a short exact sequence $0\to M\to T_1\to T_2\to 0$ for some $K$-torus $T_1, T_2$.
Let $T_1', T_2'$ be the $K$-tori whose modules of characters are the modules of cocharacters of $T_1, T_2$,
respectively.
Let $M'=[T_2'\to T_1']$ be the complex concentrated in degree $-1$ and $0$.
By \cite{Tia21PT}*{Example 4.13(2)}, we have an exact sequence
\beq\label{seq: cup product for group of mult type}
\xm{
H^i(K,M)\ar[r]^-{\Delta} & \BBP^i(K,M)\ar[r]^-{\Sigma} & H^{2-i}(K,M')^D
}
\eeq
for $i=1,3$, and an exact sequence
\beq\label{seq: cup product for group of mult type 2}
\xm{
H^2(K,M)\ar[r]^-{\Delta} & \BBP^2(K,M)_{\mrm{tors}}\ar[r]^-{\Sigma} & H^{0}(K,M')^D,
}
\eeq
where $\Sigma$ is induced by the cup-product pairing
\[
H^i(K_v,M)\times H^{2-i}(K_v,M')\to H^4(K_v,\Z(2))\simeq \QZ.
\]
Take any $\al\in \het^i(X,M)$ and any $(x_v)\in X(\BA)$.
For $i=2$, the connected-\'etale exact sequence $0\to M^{\circ}\to M\to \pi_0(M)\to 0$
yields an exact sequence of abelian groups
\[
H^2(X,M^{\circ})\to H^2(X,M)\to H^2(X,\pi_0(M)).
\]
Recall \cite{CTS21}*{Lemma 3.5.2} that $\Br(X)$ is torsion,
so a restriction-corestriction argument shows that $H^2(X,M^{\circ})$ is also torsion.
So $H^2(X,M)$ is torsion by d\'evissage.
Hence we conclude $(x_v^*\al)\in \BBP^2(X,M)_{\mrm{tors}}$.
Now due to \bref{seq: cup product for group of mult type} and \bref{seq: cup product for group of mult type 2},
$(x_v)\in X(\BA)^{\al}$ 
is equivalent to $(x_v^*\al)\in\Im\Delta= \ker\Sigma$, i.e.,
for any $\beta'\in H^{2-i}(K,M')$ with
image $\beta_v'\in H^{2-i}(K_v,M')$,
we have $\sum_{v\in C\uun} j_v'((x_v^*\al)\cup \beta_v')=0$.
By the functoriality, 
$(x_v^*\al)\in\Im\Delta$ amounts to $\sum_{v\in C\uun} j_v'\circ x_v^{\#}(\al\cup \beta'_X)=\sum_{v\in C\uun} j_v'((x_v^*\al)\cup \beta_v')=0$,
where $\beta'_X\in \het^{2-i}(X,M')$ is the image of $\beta'$.

\itm
Let $F'\ce \ul{\hom}(F,\LB)$ be the dual of $F$.
By \cite{HSS15}*{Theorem 3.3}, we obtain an exact sequence of abelian groups
\[
\xm{
H^i(K,F)\ar[r]^-{\Delta} & \BBP^i(K,F)\ar[r]^-{\Sigma} & H^{3-i}(K,F')^D
}
\]
for each $i\ge 0$,
where $\Sigma$ is given by the cup-product pairing (see \cite{HS16}*{p.~578, (10)})
\[
H^i(K_v,F)\times H^{3-i}(K_v,F')\to H^3(K_v,\LB)\simeq \QZ.
\]
Now a similar argument as in $(1)$ implies $X(\BA)^{\het^3}\subset X(\BA)^{\al}$.
\qedhere
\eenum
\end{proof}

\begin{cor}
Let $G$ be a commutative linear $K$-group and let $f:Y\to X$ be any $G$-torsor.
We have
\[
X(\BA)^{\hnr^3}\subset X(\BA)^f.
\]
\inpart we have $X(\BA)^{\hnr^3}\subset X(\BA)^{\mrm{comm}}$, 
where 
\[
X(\BA)^{\mrm{comm}}\ce \tst\bigcap\limits_{M}\bigcap\limits_{f}\bigcup\limits_{[\sg]\in H^1(K,M)}f_{\sg}(Y_{\sg}(\BA))
\]
with $M$ running over all such groups and
$f:Y\to X$ running over all $M$-torsors.
\end{cor}
\begin{proof}
Let $\rad^u(G)$ be the unipotent radical of $G$ and let $M=G/\rad^u(G)$
which is a group of multiplicative type.
Since $H^1(K,\rad^u(G))=H^2(K,\rad^u(G))=1$ are trivial (recall $\tz(K)=0$),
we deduce $H^1(K,G)\simeq H^1(K,M)$.
So we may assume that $G$ is a group of multiplicative type.
Therefore the desired assertion follows from
\cref{thm: comparing Het and various descents}\ref{thm: comparing Het and various descents 1}.
\end{proof}

We conclude this section with the following result which slightly improves the above one.

\begin{prop}
Let $G$ be a $K$-group of multiplicative type.
Let $f:Y\to X$ be a $G$-torsor
and let $f^{\#}:\het^4(X,\Z(2))\to \het^4(Y,\Z(2))$ be the induced homomorphism.
We have
\[
X(\BA)^{\ker f^{\#}}\subset X(\BA)^f.
\]
\end{prop}
\begin{proof}
Put
\[
\CK_f\ce \{ [Y]\cup a_X \zc a\in H^1(K,G')\}\subset \het^4(X,\Z(2)).
\]
Then \bref{seq: cup product for group of mult type} implies the equivalence of
$(x_v)\in X(\BA)^f$ and $(x_v)\perp \CK_f$,
i.e., $X(\BA)^{\CK_f}=X(\BA)^f$.
Since $f:Y\to X$ becomes a trivial $G$-torsor over $Y$,
the commutativity of
\[
\xm{
H^1(X,G)\ar[d] \ar@{}[r]|-{\bigtimes} & H^1(X,G')\ar[r]\ar[d] & \het^4(X,\Z(2))\ar[d]\\
H^1(Y,G) \ar@{}[r]|-{\bigtimes} & H^1(Y,G')\ar[r] & \het^4(Y,\Z(2))
}
\]
implies $\CK_f\subset \ker f^{\#}$.
Thus we conclude
$
X(\BA)^{\ker f^{\#}}\subset X(\BA)^{\CK_f}=X(\BA)^f.
$
\end{proof}

\section{Comparison via an explicit example}\label{section: BM and nr for versal torsors}
This section is devoted to the construction of a smooth integral $K$-variety $X$ such that
\[
X(\BA)^{\hnr^3}\subsetneq X(\BA)^{\het^3}=X(\BA).
\]

\begin{construction}
Let $G$ be a semi-simple simply connected group over $K$.
Let $G\to \SL_n$ be a $K$-embedding for some integer $n\ge 1$.
By enlarging $n$ if necessary,
there exists a closed subvariety $Z\subset \BBA^{n}_K$ of $\codim(Z,\BBA^n_K)\ge 3$ such that
$G$ acts freely on $Y\ce \BBA^{n}_K\setminus Z$.
So we obtain a $G$-torsor $Y\to X\ce Y/G$.
Note that $X(K)\ne \es$.
\end{construction}

We first show $\het^3(X,\LB)\simeq H^3(K,\LB)$
which implies $X(\BA)^{\het^3}=X(\BA)$.
Let $\olK$ be a fixed algebraic closure of $K$.
Write $\ol{X}\ce X\times_K\olK$, and similarly for $\ol{Y}, \ol{G}$.

\begin{lem}\label{lem: cohomology of simply connected Y}
We have $\het^i(Y,\LB)\simeq H^i(K,\LB)$ and $\het^i(\ol{Y},\LB)=0$ for $1\le i\le 3$.
\end{lem}
\begin{proof}
Since $\codim(Z,\BBA^n_K)\ge 3$,
we see that $H^i_Z(\BBA^n_K,\LB)=0$ for $i\le 4$ by purity
(see \cite{SGA4III}*{Expos\'e XIX, Th\'eor\`eme 3.2}).
Subsequently, by \cite{SGA4II}*{Expos\'e V, Proposition 6.5} there is a long exact sequence
\[
\cdots\to H^i_Z(\BBA^n_K,\LB)\to \het^i(\BBA^n_K,\LB)\to \het^i(Y,\LB)\to H^{i+1}_Z(\BBA^n_K,\LB)\to \cdots
\]
which implies $\het^i(\BBA^n_K,\LB)\simeq \het^i(Y,\LB)$ for $1\le i\le 3$.
But $\LB$ is torsion, so homotopy invariance (\cite{SGA4III}*{Expos\'e XV, Corollaire 2.2}) tells us $\het^i(\BBA^n_K,\LB)\simeq H^i(K,\LB)$.
Therefore we deduce $\het^i(Y,\LB)\simeq H^i(K,\LB)$ for $1\le i\le 3$.
Similarly, $\het^i(\ol{Y},\LB)\simeq H^i(\olK,\LB)=0$ holds for $1\le i\le 3$.
\end{proof}

Next, we show $\het^3(\ol{X},\LB)=0$ by proving that it is a subgroup of $\het^3(\ol{Y},\LB)=0$.
To this end, we introduce the \v{C}ech spectral sequence associated to the covering $\ol{Y}\to \ol{X}$.
As we shall see in the proof of \cref{lem: vanishing in Cech spec sequences},
the simply connected assumption on $G$ plays a role.

\begin{construction}
Consider the \v{C}ech spectral sequence for the covering $f:\ol{Y}\to \ol{X}$
\beq\label{eq: Cech spec seq}
E^{i,j}_2\ce \hc{i}(\ol{Y}/\ol{X},\sH_{\et}^j(\LB))\Rightarrow \het^{i+j}(\ol{X},\LB),
\eeq
where $\sH_{\et}^j(\LB)$ is the \'etale presheaf $-\mt \het^j(-,\LB)$ on $\ol{X}$.
By definition, the $E^{i,j}_2$-term of \bref{eq: Cech spec seq} is computed by the complex
$H^j(\olY^{\bullet},\LB)$,
where $\olY^0\ce \ol{Y}$ and $\olY^{d+1}\ce \olY^d\times_{\olX}\olY$.
For each $n\ge 1$, there is a canonical isomorphism of varieties
\[
\olY\times_{\olK}\ol{G}^d\simeq \olY^{d+1},\ (y,\seq{g}{d})\mt (y,y.g_1,y.g_1g_2,\dots,y.g_1\cdots g_d).
\]
Thus $E^{i,j}_2$ is the cohomology at $H^j(\olY\times_{\olK}\ol{G}^i,\LB)$
of the complex $H^j(\olY\times_{\olK}\ol{G}^{\bullet},\LB)$.
\end{construction}

\begin{lem}\label{lem: vanishing in Cech spec sequences}
Let $E^{i,j}_2$ be the $(i,j)$-term in the spectral sequence \bref{eq: Cech spec seq}.
We have $E^{i,0}_2=0$ for each $i\ge 1$, and
$E^{i,j}_2=0$ for any $i\ge 0$ and each $j=1,2$.
\end{lem}
\begin{proof}
By \cite{Cao24}*{pp.~643-644, (1)},
we have $E^{i,0}_2=0$ for each $i\ge 1$.
The \Kunneth formula
(see \cite{Cao23}*{Th\'eor\`eme 2.1 and p.~265~(iii)}) yields functorial decompositions
\[
\het^1(\olY\times_{\olK}\ol{G}^i,\mu_r^{\ots 2})
=\het^1(\olY,\mu_r^{\ots 2})\oplus \het^1(\ol{G}^i,\mu_r^{\ots 2}),
\]
and
\[
\het^2(\olY\times_{\olK}\ol{G}^i,\mu_r^{\ots 2})
=\het^2(\olY,\mu_r^{\ots 2})
\oplus \het^2(\ol{G}^i,\mu_r^{\ots 2})
\oplus \het^1(\olY,\het^1(\ol{G}^i,\mu_r^{\ots2})). 
\]
Note that
$\pi_1^{\et}(\olY)\simeq \pi_1^{\et}(\BBA^{N}_{\olK})=0$
by \cite{SGA1}*{Expos\'e X, Corollaire 3.3}
and $\pi_1^{\et}(\ol{G})=0$ because it is simply connected.
Thus \cite{Cao23}*{Th\'eor\`eme 2.1} and \cite{SGA1}*{Expos\'e XI, \S5, ($\ast$)} imply
\beq\label{eq: Kunneth formula in degree 1 for sssc G}
\het^1(\ol{G}^i,\mu_r^{\ots2})=\het^1(\ol{G},\mu_r^{\ots2})^{\ops i}=H^1(\pi_1^{\et}(\ol{G}),\mu_r^{\ots2})^{\ops i}=0,
\eeq
respectively.
Hence $\het^1(\olY,\het^1(\ol{G}^i,\mu_r^{\ots2}))=0$.
Taking direct limit and
applying \cref{lem: cohomology of simply connected Y} yield for any $i\ge 0$ and $j=1,2$ that
\[
\het^j(\olY\times_{\olK}\ol{G}^i,\LB)
=\het^j(\ol{G}^i,\LB).
\]
But $G$ is simply connected,
so $\het^1(\ol{G},\LB)=0$ by \bref{eq: Kunneth formula in degree 1 for sssc G}
and $\het^2(\ol{G},\LB)=0$ by \cite{CTX09}*{proof of Proposition 2.6}.
Therefore we deduce $\het^j(\ol{G}^i,\LB)=0$ by a further application of \cite{Cao23}*{Th\'eor\`eme 2.1} to $\ol{G}^i$.
\end{proof}

\begin{prop}\label{prop: vanishing of Hi of bar X}
For each $1\le i\le 3$, we have $\het^i(\olX,\LB)=0$.
\end{prop}
\begin{proof}
The spectral sequence \bref{eq: Cech spec seq} yields an exact sequence
\[
0\to E^{1,0}_2\to H^1\to E^{0,1}_2\to E^{2,0}_2\to \ker(H^2\to E^{0,2}_2)\to E^{1,1}_2\to E^{3,0}_2,
\]
which implies $H^1=H^2=0$ after \cref{lem: vanishing in Cech spec sequences}.
Moreover, we deduce $H^3=E^{0,3}_2$ by \cref{lem: vanishing in Cech spec sequences}.
Subsequently a direct computation 
implies $E^{0,3}_2\subset H^3(\ol{Y},\LB)=0$ and $H^3\simeq E^{0,3}_2=0$.
\end{proof}

\begin{cor}
There is an isomorphism $\het^3(X,\LB)\simeq H^3(K,\LB)$.
\end{cor}
\begin{proof}
Consider the Hochschild--Serre spectral sequence
\[
\BE^{i,j}_2\ce H^i(K,\het^j(\olX,\LB))\Rightarrow \het^{i+j}(X,\LB).
\]
\cref{prop: vanishing of Hi of bar X} yields
$\BE^{0,3}\lif=\BE^{1,2}\lif=\BE^{2,1}\lif=0$,
which implies
$H^3=F^1H^3=F^2H^3=F^3H^3\simeq \BE^{3,0}\lif=\BE^{3,0}_2=H^3(K,\LB)$.
Thus we get $\het^3(X,\LB)=H^3=F^1H^3\simeq H^3(K,\LB)$, as desired.
\end{proof}

Subsequently, we describe $\hnr^3(X,\LB)$ via cohomological invariants.
Although the following constructions concerning cohomological invariants work
for arbitrary linear algebraic groups,
we still assume that $G$ is semi-simple simply connected for context consistency.

\begin{para}
[Cohomological invariants]
Consider the following two functors
\[
H^1(-,G),\ H^d(-,\QZ(d-1)): \Field_K\to \Set
\]
from the category of extension fields of $K$ to that of sets.
Let $\Inv^d(G,\QZ(d-1))$ be the group of all natural transformations
$H^1(-,G)\to H^d(-,\QZ(d-1))$,
whose elements are called \emph{cohomological invariants} of degree $d$.
Subsequently,
let $\Inv^d(G,\QZ(d-1))_{\mrm{norm}}$ be the subgroup of invariants sending the neutral class to $0$.
For instance,
if $G$ is a connected linear group,
then $\Inv^2(G,\QZ(1))_{\mrm{norm}}\simeq \Pic(G)$ by \cite{KMRT98}*{Proposition 31.19}. 
\end{para}

Let $Y_{\eta}$ be the generic fibre of $Y\to X$ and
let $[Y_{\eta}]\in H^1(K(X),G)$ be its class.
This class defines an injective (see \cite{Mer17}*{Lemma 6.2})
evaluation map of abelian groups
\[
\Phi:\Inv^3(G,\LB)\to H^3(K(X),\LB),\ a\mt a_{K(X)}([Y_{\eta}]).
\]
Actually, $\Im\Phi$ is well understood due to the following

\begin{thm}
[\cite{GMS03}*{pp.~99-100, Appendix C} or \cite{Mer17}*{Theorem 6.3}]\label{thm: Inv3 is Hnr3}
There is an isomorphism of abelian groups
\[
\Phi:\Inv^3(G,\LB)\to H^0_{\Zar}(X,\CH^3_{\et}(\LB))=\hnr^3(X,\LB).
\]
\end{thm}

\begin{para}\label{para: converse of evaluation at gen torsor}
We recall the construction of the inverse of $\Phi$, denoted by
\beq\label{eq: inverse of Inv3 to Hnr3 map}
\Psi:\hnr^3(X,\LB)\to \Inv^3(G,\LB).
\eeq
For any $\beta\in \hnr^3(X,\LB)$ and
any field extension $L/K$,
let $\Psi(\beta)_L:H^1(L,G)\to H^3(L,\LB)$ be the map to be defined.
Let $[E]\in H^1(L,G)$ be the class of a $G$-torsor.
Consider the diagram
\beac\label{diag: inverse of Inv3 to Hnr3 map}
\xm{
\e L\simeq E/G_L & (E\times_LY_L)/G_L\ar[l]_-{f_L}\ar[r]\ar@/^1.5pc/[rr]^-{q_L} & Y_L/G_L\simeq X_L\ar[r] & X.
}
\eeac
By homotopy invariance \cite{Mer16red}*{p.~702}, $f_L^*:H^3(L,\LB)\to \hnr^3((E\times_LY_L)/G_L,\LB)$
is an isomorphism.
Thus we are allowed to put
\[
\Psi(\beta)_L([E])\ce (f_L^*)\ui\circ q_L^*(\beta).
\]
\end{para}

We close this section by the promised comparison $X(\BA)^{\hnr^3}\subsetneq X(\BA)$.

\begin{lem}\label{lem: formula for Hnr evaluation}
For any field extension $L/K$ and any $x\in X(L)$,
let $Y_x$ be the fibre above the image of $x:\e L\to X$ and
let $\tilde{x}^{\#}:\hnr^3(X,\LB)\to H^3(L,\LB)$ be the induced map.
For any $\beta\in \hnr^3(X,\LB)$, we have
\[
\tilde{x}^{\#}(\beta)=\Psi(\beta)_L([Y_x]).
\]
\end{lem}
\begin{proof}
Let $Y_{\eta}$ be the generic fibre of $Y\to X$.
Let $f_L$ and $q_L$ be as in diagram \bref{diag: inverse of Inv3 to Hnr3 map}.
Consider the following \cmdm of abelian groups
\[
\xm{
H^0_{\Zar}(X,\CH^3_{\et}(\LB))\ar[r]^-{\pr_1^*}\ar[d]_-{\tilde{x}^{\#}} 
& H^0_{\Zar}(Y\times_KY/G,\CH^3_{\et}(\LB))\ar[d]_-{(x\times\id)^*} 
& H^0_{\Zar}(X,\CH^3_{\et}(\LB))\ar[l]_-{\pr_2^*}\ar@{=}[d]\\ 
H^3(L,\LB)\ar[r]^-{f_L^*} & H^0_{\Zar}(Y_x\times_LY_L/G_L,\CH^3_{\et}(\LB)) & H^0_{\Zar}(X,\CH^3_{\et}(\LB)).\ar[l]_-{q_L^*}
}
\]
By \cite{BM13}*{Paragraph 3a and Theorem 3.4},
we conclude $\pr_1^*(\beta)=\pr_2^*(\beta)$.
Subsequently, we obtain $f_L^*\circ \tilde{x}^{\#}(\beta)=q_L^*(\beta)$,
i.e., $\tilde{x}^{\#}(\beta)=(f_L^*)\ui\circ q_L^*(\beta)=\Psi(\beta)_L([Y_x])$.
\end{proof}

\begin{para}\label{para: Rost invariant and an example from Hu}
Let $G$ be an absolutely almost-simple simply connected group over $K$.
According to \cite{KMRT98}*{Proposition 31.40}
the group $\Inv^3(G,\LB)_{\mrm{norm}}$ is cyclic of order $n_G$
with $n_G$ a constant determined by the type of $G$.
Let $\s{R}\in \Inv^3(G,\LB)_{\mrm{norm}}$ be the \emph{Rost invariant},
i.e., the canonical generator of $\Inv^3(G,\LB)_{\mrm{norm}}$.
For each $v\in C\uun$, let $\s{R}_v:H^1(K_v,G)\to H^3(K_v,\LB)$ and
define 
\[
\ker \s{R}_v\ce \{\al_v\in H^1(K_v,G)\zc \s{R}_v(\al_v)=0\}.
\]

In view of Serre's conjecture II, the sets $H^1(K_v,G)$ are typically non-trivial because of $\mrm{cd}(K_v)=3$.
On the other hand, the Rost kernel is known to be trivial for many groups by the works of \cites{Gil00, Gar01, CTPS12, Pre13, PPS18}.
Now we give an example satisfying the assumptions of \cref{thm: comparison of Hnr and Het via abs assc groups} below
which is communicated to the authors by Yong HU.

Take a hyperbolic quadratic form $q$ of dimension $8$ over $K$,
which has trivial discriminant and trivial Clifford invariant by \cite{KMRT98}*{Corollary 35.3}.
Then $H^1(K_v,\Spin(q))$ classifies quadratic forms over $K_v$ of dimension $8$
with trivial discriminant and Clifford invariant.
Since for each $v\in C\uun$ the $u$-invariant of $K_v$ is $8$ by \cite{PS10}*{Theorem 4.6} and \cite{Lee13}*{Theorem 3.4},
the existence of an anisotropic $3$-Pfister form over $K_v$
yields a non-trivial element of the set $H^1(K_v,\Spin(q))$.
Finally, $\ker\s{R}_v$ is trivial since $\Spin(q)$ is split (see \cite{Gar01}*{p.~686}).
\end{para}

\begin{thm}\label{thm: comparison of Hnr and Het via abs assc groups}
Let $G$ be an absolutely almost-simple simply connected $K$-group.
Suppose $H^1(K_v,G)\ne 1$ and $\ker\s{R}_v=1$ for infinitely many $v$.
Then we have
\[
X(\BA)^{\hnr^3}\subsetneq X(\BA).
\]
\end{thm}
\begin{proof}
Let $\Psi:\hnr^3(X,\LB)\to \Inv^3(G,\LB)$ be the homomorphism \bref{eq: inverse of Inv3 to Hnr3 map}.
Choose $\beta \in \hnr^3(X,\LB)$ such that $\Psi(\beta)=\s{R}$.
Take any $(x_v)\in X(\BA)^{\hnr^3}$, i.e., $\sum j_v\circ \tilde{x}_v^{\#}(\beta)=0$
(see \bref{eq: BM pairing for Hnr simpler}),
where $\tilde{x}_v^{\#}:\hnr^3(X,\LB)\to H^3(K_v,\LB)$ is the induced map.
However for all but finitely many $v\in C\uun$,
we have $\tilde{x}_v^{\#}(\beta)=0$ (see \ref{para: unramified cohomology}\ref{para: unramified cohomology 3}),
thus we may choose $w\in C\uun$ such that
$\tilde{x}_w^{\#}(\beta)=0$, $H^1(K_w,G)\ne 1$ and $\ker\s{R}_w=1$.

Since $\dif_w:X(K_w)\to H^1(K_w,G)$ is surjective because of $H^1(K_w,\SL_n)=1$,
there exists $z_w\in X(K_w)$ such that $\dif_w(z_w)=[Y_{z_w}]\notin \ker\s{R}_{w}$.
Then \cref{lem: formula for Hnr evaluation} tells us
\[
\tilde{z}_w^{\#}(\beta)=\Psi(\beta)_{K_w}([Y_{z_w}])=\s{R}_w([Y_{z_w}])\ne 0.
\]
If we write $z_v=x_v$ for $v\ne w$,
then $\sum j_v\circ \tilde{z}_v^{\#}(\beta)=j_w\circ \tilde{z}_w^{\#}(\beta)\ne 0$,
i.e., $(z_v)\notin X(\BA)^{\hnr^3}$.
\end{proof}

\section{Unramified obstructions and descents}\label{section: unramified obs with general descent}
This section aims to comparing the unramified obstruction with the descent obstruction.
It is well-known $\hnr^3(K(\SL_n)/K,\LB)\simeq H^3(K,\LB)$ since $\SL_n$ is rational
(for instance, see \cite{Mer02}*{p.~446}).
However, we still have $\hnr^3(\SL_{n},\LB)\simeq H^3(K,\LB)$ for distinct reasons.

\begin{lem}\label{cor: Inv3 of SLn is constant}
There is an isomorphism $\hnr^3(\SL_{n},\LB)\simeq H^3(K,\LB)$ for any $n\ge 2$.
\end{lem}
\begin{proof}
Applying \cite{EKLV98}*{Assumption 3.19 and Proposition 3.20(i)} to $G=H=\SL_n$ yields 
the desired isomorphism.
\end{proof}

\begin{para}
Let $G$ be a linear $K$-group and embed it into $\SL_n$ for some $n\ge 2$.
Let $X=\SL_n/G$. 
Take any $x\in X(K)$ and 
let $q_x:\SL_n\to X$, $s\mt s.x$.
By the functoriality of cup-product,
we see that
\[
s\in\SL_n(\BA)^{q_x^*(\hnr^3(X))}\iff q_x(s)\in X(\BA)^{\hnr^3(X)}.
\]
But $q_x^*(\hnr^3(X))\subset \hnr^3(\SL_n,\LB)\simeq H^3(K,\LB)$,
so we conclude $q_x(s)=s.x\in X(\BA)^{\hnr^3(X)}$ for any $s\in \SL_n(\BA)$, i.e.,
$\SL_n(\BA).X(K)\subset X(\BA)^{\hnr^3(X)}$.
If its converse holds, we prove that the unramified obstruction is finer than the descent obstruction.
\end{para}

\begin{lem}\label{lem: unramified third cohomology of YZ are iso}
Let $Z$ be a smooth integral $K$-variety and
let $f:Y\to Z$ be a $G$-torsor.
Let ${_n}Y\ce \SL_n\times^GY$ be the contracted product and
let $f_n:{_n}Y\to Z$ be the induced $\SL_n$-torsor.
Then the induced map
\[
f_n^*:\hnr^3(Z,\LB)\to \hnr^3({_n}Y,\LB)
\]
is an isomorphism of abelian groups.
\end{lem}
\begin{proof}
Let $\eta$ be the generic point of $Z$ and
let ${_n}Y_{\eta}$ be the generic fibre of ${_n}Y\to Z$.
For any $y\in Y$ and any $z\in Z$,
let $\kappa_y$ and $\kappa_z$ be their respective residue fields. 
Then $f_n:{_n}Y\to Z$ induces a \cmdm
\beac\label{diag: unramified coh sequence}
\xm{
0\ar[r] & \hnr^3(Z,\LB)\ar[r]\ar[d] & H^3(\kappa_{\eta},\LB)\ar[r]\ar[d] & \bigoplus\limits_{z\in Z\uun}H^2(\kappa_z,\LB(-1))\ar[d] \\
0\ar[r] & \hnr^3({_n}Y,\LB)\ar[r] & \hnr^3({_n}Y_{\eta},\LB)\ar[r] & \bigoplus\limits_{z\in Z\uun}\bigoplus\limits_{y\in {_n}Y_z^{(0)}}H^2(\kappa_y,\LB(-1)).
}
\eeac
The upper row is exact by \cite{CT95}*{Theorem 4.1.1} and
the exactness of the lower row follows from the exactness of
\[
0\to \hnr^3({_n}Y,\LB)
\to \ker\big(H^3(K({_n}Y),\LB)\to \tst\bigoplus\limits_{y \in {_n}Y\uun_{\eta}} H^2(\kappa_y,\LB(-1))\big)
\to \tst\bigoplus\limits_{y\in {_n}Y\uun} H^2(\kappa_y,\LB(-1))
\]
whose middle term is nothing but $\hnr^3({_n}Y_{\eta},\LB)$ by \cite{CT95}*{Theorem 4.1.1} again.
But ${_n}Y_{\eta}\to \kappa_{\eta}$ is an $\SL_n$-torsor,
so ${_n}Y_{\eta}\simeq \SL_n$ as $\kappa_{\eta}$-varieties (by $H^1(\kappa_{\eta},\SL_n)=1$).
So we deduce
\[
\hnr^3({_n}Y_{\eta},\LB)
\simeq \hnr^3(\SL_n,\LB)
\simeq H^3(\kappa_{\eta},\LB).
\]
By a diagram chase, we reduce to prove the injectivity of the right vertical arrow in \bref{diag: unramified coh sequence}.

Take any $z\in Z\uun$.
The fibre ${_n}Y_z\to \e\kappa_z$ is then an $\SL_n$-torsor,
so ${_n}Y_z\simeq \SL_n$ as $\kappa_z$-varieties.
Hence ${_n}Y_z^{(0)}$ consists of the generic point of ${_n}Y_z$. 
Recall \bref{eq: Kummer sequence H2 and Br}
and note that
\[
H^2(\kappa_z,\LB(-1))\to \tst\bigoplus\limits_{y\in {_n}Y_z^{(0)}}H^2(\kappa_y,\LB(-1))=H^2(\kappa_z({_n}Y_z),\LB(-1))
\]
factors as $\Br(\kappa_z)\to \Br({_{n}}Y_z)\to \Br(\kappa_z({_{n}}Y_z))$
whose former map is injective since ${_n}Y_z(\kappa_z)\simeq \SL_n(\kappa_z)\ne\es$ and
latter one is injective by \cite{CTS21}*{Theorem 3.5.5}. 
%
\end{proof}

\begin{thm}\label{thm: compare hnr and descent}
Let $G$ be a linear $K$-group and fix a $K$-embedding $G\to \SL_n$ for some $n\ge 2$.
The following are equivalent.
\benumr
\itm\label{thm: compare hnr and descent 1}
Let $X=\SL_n/G$.
There is an identification of sets $X(\BA)^{\hnr^3}={\SL_n(\BA).X(K)}$.

\itm\label{thm: compare hnr and descent 2}
Let $Z$ be any smooth integral variety over $K$.
For any $G$-torsor $f:Y\to Z$,
there is an inclusion $Z(\BA)^{\hnr^3}\subset Z(\BA)^f$.
\eenum
\end{thm}
\begin{proof}
We write $\f{X}(\BA)^{\hnr^3(\f{X})}$ instead of $\f{X}(\BA)^{\hnr^3(\f{X},\LB)}$ to avoid heavy notations.
\bitem
\itm
Assume \ref{thm: compare hnr and descent 1}.
Let ${_n}Y\ce \SL_n\times^GY$ be the contracted product.
Let $f_n:{_n}Y\to Z$ be the induced $\SL_n$-torsor and
let $\pi:{_n}Y\to X$ be the projection. 
By \cref{lem: unramified third cohomology of YZ are iso} and functoriality,
we obtain
\[
{_n}Y(\BA)^{\hnr^3({_n}Y)}=f_n\ui(Z(\BA)^{\hnr^3(Z)}).
\]
It follows that
\[
Z(\BA)^{\hnr^3(Z)}
=f_n({_n}Y(\BA)^{\hnr^3({_n}Y)})
\subset f_n({_n}Y(\BA)^{\pi^*\hnr^3(X)})
\subset f_n(\pi\ui(X(\BA)^{\hnr^3(X)})).
\]
Now take any $z\in Z(\BA)^{\hnr^3(Z)}$ and
suppose $z=f_n(y)$ for some $y\in {_n}Y(\BA)$ such that $\pi(y)\in X(\BA)^{\hnr^3(X)}$.
By assumption, we may assume $\pi(y)=s.x$ for some $s\in \SL_n(\BA)$ and some $x\in X(K)$.
So we obtain $s\ui.y\in {_n}Y_x(\BA)$ and $f_n(s\ui.y)=f_n(y)=z$.
Taking the following exact sequence of pointed sets
\[
1\to G(K)\to \SL_n(K)\stra{q} X(K)\stra{\dif} H^1(K,G)\to 1
\]
into account, we conclude
${_n}Y_x\simeq q\ui(x)\times^GY\simeq {_{\sg}}Y$ with $[\sg]=\dif(x)\in H^1(K,G)$.
Thus $z\in {_{\sg}}f({_{\sg}}Y(\BA))\subset Z(\BA)^f$, as desired.

%

\itm
Conversely, take $Y=\SL_n$ and $Z=X$.
It suffices to show $X(\BA)^f=\SL_n(\BA).X(K)$.
Consider the exact sequence of pointed sets
\[
\SL_n(K)\stra{f} X(K)\stra{\dif} H^1(K,G)\to 1.
\]
Take any $[\sg]\in H^1(K,G)$ and
say $[\sg]=\dif(x)$ for some $x\in X(K)$.
Note that ${_{\sg}}\SL_n$ is an $\SL_n$-torsor
because ${_{\sg}}\SL_n\ce \SL_n\times^Gf\ui(x)$ commutes with the multiplication of $\SL_n$ from the left.
But $H^1(K,\SL_n)=1$,
so there is an isomorphism
$\lb_{\sg}:{_{\sg}}\SL_n\simeq \SL_n$ of $K$-varieties
which implies ${_{\sg}}f(y)=\lb_{\sg}(y).x$ for any $y\in {_{\sg}}\SL_n(\BA)$.
Hence we obtain $X(\BA)^f=\SL_n(\BA).X(K)$.
\qedhere
\eitem
\end{proof}

\small
\noindent Yisheng TIAN \\
Institute for Advanced Study in Mathematics \\
Harbin Institute of Technology \\
Harbin 150001, China \\
Email: tysmath@mail.ustc.edu.cn

\end{document}

%% file: intro.tex
\usepackage[alphabetic,lite]{amsrefs}                                          
\usepackage{amssymb}                                                           
\usepackage{amsthm}                                                            
\usepackage{appendix}                                                          
\usepackage{array}                                                             
\usepackage{bm}                                                                
\usepackage{colonequals}                                                       
\usepackage{enumitem}                                                          
\usepackage{fullpage}                                                          
\usepackage{graphicx}                                                          
\usepackage{indentfirst}                                                       
\usepackage{mathrsfs}                                                          
\usepackage{mathtools}                                                         
\usepackage{moreenum}                                                          
\usepackage{multirow}
\usepackage{pdfpages}
\usepackage{stmaryrd}                                                          
\usepackage{titling}
\usepackage{verbatim}                                                          
\usepackage[all,cmtip]{xy}                                                     
\usepackage{tikz}
\usepackage{xcolor}                                                            

\definecolor{tianred}{rgb}{0.57, 0.36, 0.51}                                   
\definecolor{tianblue}{rgb}{0.0, 0.22, 0.66}                                   
\definecolor{tianpink}{rgb}{0.88, 0.56, 0.59}                                  
\definecolor{tiangreen}{rgb}{0.24, 0.82, 0.44}                                 
\definecolor{Prune}{RGB}{99,0,60}

\usepackage[colorlinks,
            linkcolor=tianred,
            anchorcolor=tiangreen,
            citecolor=tianblue,
            urlcolor=tianpink]{hyperref}                                      
\usepackage{cleveref}                                                          

\makeatletter
\DeclareRobustCommand\widecheck[1]{{\mathpalette\@widecheck{#1}}}
\def\@widecheck#1#2{%
    \setbox\z@\hbox{\m@th$#1#2$}%
    \setbox\tw@\hbox{\m@th$#1%
       \widehat{%
          \vrule\@width\z@\@height\ht\z@
          \vrule\@height\z@\@width\wd\z@}$}%
    \dp\tw@-\ht\z@
    \@tempdima\ht\z@ \advance\@tempdima2\ht\tw@ \divide\@tempdima\thr@@
    \setbox\tw@\hbox{%
       \raise\@tempdima\hbox{\scalebox{1}[-1]{\lower\@tempdima\box
\tw@}}}%
    {\ooalign{\box\tw@ \cr \box\z@}}}
\makeatother

\newcommand{\beac}{\begin{equation}\begin{array}{c}}                           
\newcommand{\eeac}{\end{array}\end{equation}}                                  

\newcommand{\beq}{\begin{equation}}
\newcommand{\eeq}{\end{equation}}

\newcommand{\brmks}{\begin{rmk}\hfill\benuma}
\newcommand{\ermks}{\eenum\end{rmk}}

\newcommand{\benum}{\begin{enumerate}[label={{\upshape(\alph*)}}]}             
\newcommand{\benuma}{\begin{enumerate}[label={{\upshape(\arabic*)}}]}          
\newcommand{\benumr}{\begin{enumerate}[label={{\upshape(\roman*)}}]}           
\newcommand{\eenum}{\end{enumerate}}
\newcommand{\bitem}{\begin{itemize}}                                           
\newcommand{\eitem}{\end{itemize}}                                             

\newcommand{\xm}{\xymatrix}                                                    

\newcommand{\tst}{\textstyle}                                                  
\newcommand{\tstsl}{\tst\sum\limits}                                           
\newcommand{\tstpl}{\tst\prod\limits}                                          
\newcommand{\tstops}{\tst\bigoplus\limits}                                     

\theoremstyle{plain}      \newtheorem{thm}{Theorem}[section]                   %
\theoremstyle{plain}      \newtheorem{lem}[thm]{Lemma}                         %
\theoremstyle{plain}      \newtheorem{cor}[thm]{Corollary}                     %
\theoremstyle{plain}      \newtheorem{prop}[thm]{Proposition}                  %
\theoremstyle{plain}                   %
\theoremstyle{definition} \newtheorem{rmk}[thm]{Remark}                        %
\theoremstyle{definition}                      %
\theoremstyle{definition}                       %
\theoremstyle{definition}                         %
\theoremstyle{definition}                        %
\theoremstyle{definition}                    %
\theoremstyle{definition}                        %
\theoremstyle{definition}                    %
\theoremstyle{definition}                  %
\theoremstyle{definition} \newtheorem{construction}[thm]{Construction}         %
\theoremstyle{definition}              %
\theoremstyle{definition}                  %
\theoremstyle{definition} \newtheorem{prop-df}[thm]{Proposition-Definition}    %

\theoremstyle{plain}                   %
\theoremstyle{plain}                             %
\theoremstyle{plain}                         %
\theoremstyle{definition}                      %
\theoremstyle{definition}                   %
\theoremstyle{definition}                        %
\theoremstyle{definition}                      %
\theoremstyle{definition}                  %
\theoremstyle{definition}                       %

\theoremstyle{definition}
\newtheorem*{construction*}{Construction}                                      %
\newtheorem*{conjecture*}{Conjecture}                                          %
\newtheorem*{hypothesis*}{Hypothesis}                                          %
\newtheorem*{convention*}{Convention}                                          %
\newtheorem*{notation*}{Notation}                                              %
\newtheorem*{prop*}{Proposition}                                               %
\newtheorem*{summary*}{Summary}                                                %
\newtheorem*{qt*}{Question}                                                    %
\newtheorem*{rmk*}{Remark}                                                     %
\newtheorem*{fact*}{Fact}                                                      %
\newtheorem*{lizi*}{Example}                                                   %
\newtheorem*{df*}{Definition}                                                  %
\theoremstyle{plain}
\newtheorem*{thm*}{Theorem}                                                    %
\newtheorem*{lem*}{Lemma}                                                      %
\newtheorem*{axiom*}{Axiom}                                                    %

\newtheoremstyle{subsection-tweak}
{}{}{}{}{\bfseries}{}{.5em}{\thmnumber{(\@{#1}{}\@{#2}).}\thmnote{~{\bfseries#3.}}}
\theoremstyle{subsection-tweak}
\newtheorem{para}[thm]{}

\crefname{thm}{\textup{Theorem}}{Theorems}                                     %

\crefname{lem}{\textup{Lemma}}{Lemmas}
\crefname{eg}{\textup{Example}}{Examples}
\crefname{ex}{\textup{Exercise}}{Exercise}
\crefname{rmk}{\textup{Remark}}{Remarks}
\crefname{cor}{\textup{Corollary}}{Corollaries}
\crefname{df}{\textup{Definition}}{Definitions}
\crefname{question}{\textup{Question}}{Questions}
\crefname{prop}{\textup{Proposition}}{Propositions}
\crefname{conjecture}{\textup{Conjecture}}{Conjectures}
\crefname{convention}{\textup{Convention}}{Conventions}
\crefname{construction}{\textup{Construction}}{Constructions}
\crefname{hypo}{\textup{Hypothesis}}{Hypothesis}

\crefname{thme}{Th\'eo\`eme}{Th\'eo\`emes}
\crefname{lemme}{Lemme}{Lemmes}
\crefname{ege}{Exemple}{Exemples}
\crefname{core}{Corollaire}{Corollaires}
\crefname{rmke}{Remarque}{Remarques}
\crefname{dfe}{D\'efinition}{D\'efinitions}

\newcommand{\bref}[1]{\textup{(\ref{#1})}}

\newcommand{\al}{\alpha}                                                       
\newcommand{\sg}{\sigma}                                                       
\newcommand{\lb}{\lambda}                                                      
\newcommand{\LB}{\Lambda}                                                      

\usepackage[OT2,T1]{fontenc}
\DeclareSymbolFont{cyrletters}{OT2}{wncyr}{m}{n}
\DeclareMathSymbol{\RBe}{\mathalpha}{cyrletters}{"42}                          
\DeclareMathSymbol{\Che}{\mathalpha}{cyrletters}{"51}                          
\DeclareMathSymbol{\Sha}{\mathalpha}{cyrletters}{"58}                          

\newcommand{\f}{\mathfrak}
\newcommand{\s}{\mathscr}
\newcommand{\mbf}{\mathbf}
\newcommand{\mrm}{\mathrm}

\newcommand{\BA}{\mathbf{A}}   
\newcommand{\BC}{\mathbf{C}}   
\newcommand{\BE}{\mathbf{E}}   \newcommand{\BF}{\mathbf{F}}

   \newcommand{\BR}{\mathbf{R}}

  \newcommand{\CF}{\mathcal{F}}
  \newcommand{\CH}{\mathcal{H}}
  
\newcommand{\CK}{\mathcal{K}}  
\newcommand{\CM}{\mathcal{M}}  
\newcommand{\CO}{\mathcal{O}}

  \newcommand{\CX}{\mathcal{X}}

\newcommand{\BBA}{\mathbb{A}}

\newcommand{\BBG}{\mathbb{G}}

  \newcommand{\BBP}{\mathbb{P}}

  \newcommand{\sH}{\mathscr{H}}

\newcommand{\lif}{_{\infty}}                                                   

\newcommand{\ui}{^{-1}}                                                        
\newcommand{\uun}{^{(1)}}                                                      

\newcommand{\ol}{\overline}                                                    
\newcommand{\ul}{\underline}                                                   
\newcommand{\wc}{\widecheck}                                                   

\newcommand{\Z}{\mathbb{Z}}                                                    
\newcommand{\Q}{\mathbb{Q}}                                                    

\newcommand{\Qp}{\mathbb{Q}_{p}}                                               
\newcommand{\QZ}{\mathbb{Q}/\mathbb{Z}}                                        

\newcommand{\olK}{\overline{K}}                                                
\newcommand{\olX}{\overline{X}}                                                
\newcommand{\olY}{\overline{Y}}                                                

\newcommand{\lip}{\langle}                                                     
\newcommand{\rip}{\rangle}                                                     

\DeclareMathOperator{\id}{id}                                                  

\DeclareMathOperator{\tz}{char}                                                

\DeclareMathOperator{\ops}{\oplus}                                             
\DeclareMathOperator{\ots}{\otimes}                                            

\DeclareMathOperator{\codim}{codim}                                            

\DeclareMathOperator{\e}{Spec}                                                 

\newcommand{\Set}{\mathfrak{Set}}                                              
\newcommand{\Field}{\mathfrak{Field}}                                          
\newcommand{\Sch}{\mathfrak{Sch}}                                              

\DeclareMathOperator{\Hom}{Hom}                                                
\renewcommand{\hom}{\Hom}                                                      

\newcommand{\drl}{\varinjlim}                                                  

\DeclareMathOperator{\Ker}{Ker}                                                
\renewcommand{\ker}{\Ker}                                                      
\DeclareMathOperator{\Image}{Im}                                               
\renewcommand{\Im}{\Image}                                                     

\DeclareMathOperator{\dif}{\partial}                                           

\DeclareMathOperator{\Pic}{Pic}                                                
\DeclareMathOperator{\Br}{Br}                                                  

\DeclareMathOperator{\rmCH}{CH}                                                

\DeclareMathOperator{\Inv}{Inv}                                                

\newcommand{\gm}{\BBG_m}                                                       
\newcommand{\ga}{\BBG_a}                                                       
\DeclareMathOperator{\SL}{\mathbf{SL}}                                         
\DeclareMathOperator{\Spin}{\mathbf{Spin}}                                     

\DeclareMathOperator{\rad}{rad}                                                

\newcommand{\Zar}{\mathrm{Zar}}                                                
\newcommand{\et}{\mathrm{\acute{e}t}}                                          

\newcommand{\nr}{{\mathrm{nr}}}                                                

\newcommand{\het}{H_{\et}}                                                     
\newcommand{\hc}[1]{\wc{H}{^{#1}}}                                             

\newcommand{\hnr}{H_{\nr}}                                                     

\newcommand{\stra}[1]{\stackrel{#1}{\to}}                                      
\newcommand{\mt}{\mapsto}                                                      

\newcommand{\qand}{\quad\ \textup{and}\quad\ }                                 
\newcommand{\vem}{\vspace{1em}}                                                
\newcommand{\zc}{\,|\,}                                                        

\newcommand{\es}{\varnothing}                                                  
\DeclareMathOperator{\pr}{pr}                                                  

\newcommand{\ce}{\colonequals}                                                 

\newcommand{\seq}[2]{{#1}_1,\dots,{#1}_{{#2}}}                                 

\newcommand{\inpart}{In particular,~}                                          
\newcommand{\ttiff}{if and only if~}                                           
\newcommand{\wrt}{with respect to~}                                            
\newcommand{\cmdm}{commutative diagram~}                                       

\newcommand{\Kunneth}{K{\"u}nneth~}                                            

\newcommand{\itm}{\item}                                                       

\DeclareMathOperator{\ev}{ev}                                                  

